\documentclass[12pt]{amsart}
\usepackage{amsmath,amsthm,amssymb,mathrsfs,stmaryrd,color}

\usepackage[latin1]{inputenc}
\usepackage[arrow, matrix, curve]{xy}
\usepackage[a4paper,text={160mm,240mm},centering,headsep=5mm,footskip=10mm]{geometry}

\usepackage{amscd}
\usepackage{a4wide}
\usepackage{amsfonts}
\usepackage[T1]{fontenc}
\usepackage[latin1]{inputenc}
\usepackage[english]{babel}
\usepackage{datetime}

\numberwithin{equation}{section}
\newtheorem{theorem}{Theorem}[section]
\newtheorem{lemma}[theorem]{Lemma}

\newtheorem{proposition}[theorem]{Proposition}
\newtheorem{proposition/definition}[theorem]{Proposition/Definition}
\newtheorem{corollary}[theorem]{Corollary}
\newtheorem{definition}[theorem]{Definition}
\newtheorem{remark}[theorem]{Remark}
\newtheorem{example}[theorem]{Example}

\newtheorem*{notat*}{Notation}

\newtheorem{conj}[theorem]{Conjecture}

\renewcommand{\H}{\mathcal{H}}
\newcommand{\N}{\mathbb{N}} 
\newcommand{\Z}{\mathbb{Z}}

\renewcommand{\O}{\mathcal{O}}
\renewcommand{\r}{\rightarrow}

\newcommand{\Zar}{\mathrm{Zar}}
\newcommand{\Nis}{\mathrm{Nis}}
\newcommand{\et}{\mathrm{\acute{e}t}}
\newcommand{\Spec}{\mathrm{Spec}}

\newcommand{\CH}{\mathrm{CH}}

\renewcommand{\log}{\mathrm{log}}

\title{Algebraization for zero-cycles and the $p$-adic cycle class map}
\author{Morten L\"uders}
\address{Fakult\"at f\"ur Mathematik, Universit\"at Regensburg, 93040 Regensburg, Germany}
\email{mortenlueders@yahoo.de}

\begin{document}
\begin{abstract}
Using an idelic argument and assuming the Gersten conjecture for Milnor K-theory, we show that the restriction map from the Chow group of one-cycles on a smooth projective scheme over a henselian local ring to a pro-system of thickened zero-cycles is surjective. We relate this restriction map to the $p$-adic cycle class map.
\end{abstract}

\thanks{The author is supported by the DFG through CRC 1085 \textit{Higher Invariants} (Universit\"at Regensburg).}
\maketitle

\section{Introduction}
Let $A$ be an excellent henselian discrete valuation ring with uniformising parameter $\pi$ and residue field $k$. Let $X$ be a smooth projective scheme over Spec$(A)$ of relative dimension $d$. Let $X_n:=X\times_A A/(\pi^n)$, i.e. $X_1$ is the special fiber and the $X_n$ are the respective thickenings of $X_1$. 

For $n$ invertible in $k$ and $\Lambda=\Z/n\Z$ the following commutative diagram has been studied extensively:
\begin{equation}\label{diagrammintroduction}
\begin{xy} 
  \xymatrix{
  \CH_1(X)/n \ar[r]^{\rho} \ar[d]_{cl_X} & \CH_0(X_1)/n \ar[d]_{cl_{X_s}} \\
 H^{2d}_{\text{\'et}}(X,\Lambda(d))  \ar[r]^-{\cong} & H^{2d}_{\text{\'et}}(X_1,\Lambda(d))   
  }
\end{xy} 
\end{equation}
The lower horizontal map is an isomorphism by proper base change. The map $cl_{X_s}$ is an isomorphism assuming that $k$ is finite or separably closed by unramified class field theory (see \cite[Thm. 5, Rem. 3]{CSS83} and \cite{KaS83}). In \cite{SS10}, Saito and Sato show that $cl_X$ is an isomorphism if $k$ is finite or separably closed which implies that $\rho$ is an isomorphism under these conditions. That $\rho$ is in fact an isomorphism for arbitrary perfect residue fields is shown in \cite{KEW16} without using \'etale realizations by making use of a method introduced by Bloch in \cite[App.]{EWB16}. In \cite{Lu17} these results are generalized to zero-cycles with coefficients in  Milnor K-theory.

Let $\mathcal{K}^M_{X,d}$ be the improved Milnor K-sheaf defined in \cite{Ke10} and $\mathcal{K}^M_{d,X_n}$ its restriction to $X_n$. In this article we study the restriction map
$$\begin{xy}
  \xymatrix{
       res_{X_n}: \CH^{d+j}(X,j) \ar[r]^-{\cong}   &  
    H^d(X,\mathcal{K}^M_{d+j,X}) \ar[r]^-{res_{X_n}} &  H^d(X_1,\mathcal{K}^M_{d+j,X_n}).
  }
\end{xy} $$
If $j=0$, we assume the Gersten conjecture for the Milnor K-sheaf $\mathcal{K}^M_{n,X}$ (see Definition \ref{Gcdef}) for the isomorphism on the left. By  \cite{Ke09} and \cite{Ke10} it holds if $X$ is equi-characteristic. If $j>0$, we additionally assume the Gersten conjecture for the sheaf $\mathcal{CH}^r(-q)$ associated to the presheaf $U\mapsto \CH^r(U,-q)$ for the isomorphism on the left (see \cite{Lu17}). This holds with finite coefficients or again if $X$ is equi-characteristic (see e.g. \cite{Lu17}). For our applications we will need the following additional result which is well-known to the expert and easily follows from what is known about the Gersten conjecture for Quillen K-theory with finite coefficients (see Section \ref{sectionGCMK}):
\begin{proposition}
Let $\mathrm{ch}(k)=p>(n-1)$. Then the Gersten conjecture holds for the sheaf $\mathcal{K}^M_{n,X}/p^r$ for all $r\geq 1$. 
\end{proposition}

The main theorem of this article is the following:
\begin{theorem}\label{maintheoremalgebraization}
The restriction map $res:H^{d}(X,\mathcal{K}^M_{d+j,X})\xrightarrow{} H^{d}(X_1,\mathcal{K}^M_{d+j,X_n})$ is surjective. In particular the map of pro-systems 
$$res: H^{d}(X,\mathcal{K}^M_{d+j,X}) \r "\mathrm{lim}_n" H^{d}(X_1,\mathcal{K}^M_{d+j,X_n})$$
is an epimorphism in the category of pro-abelian groups $\text{pro-}\text{Ab}$ for all $j\geq 0$. Here we consider $H^{d}(X,\mathcal{K}^M_{d+j,X})$ as a constant pro-system.
\end{theorem}
Theorem \ref{maintheoremalgebraization} is a partial response to a conjecture by Kerz, Esnault and Wittenberg saying that assuming the Gersten conjecture for the Milnor K-sheaf $\mathcal{K}^M_{n,X}$ the restriction map $res: \CH^{d}(X)\otimes \Z/p^r\Z \r "\text{lim}_n" H^{d}(X_1,\mathcal{K}^M_{d,X_n}/p^r)$ is an isomorphism if $\text{ch}(\text{Quot}(A))=0$ and if $k$ is perfect of characteristic $p>0$ (see \cite[Sec. 10]{KEW16}). The proof of Theorem \ref{maintheoremalgebraization} relies on so-called idelic arguments. A different approach using differential forms is explained in \cite{Lu18}.

Let again $\text{ch}(\text{Quot}(A))=0$ and $k$ be perfect of characteristic $p>0$. In the final section of this article we relate the restriction map $res$ to the $p$-adic cycle class map $\varrho_{p^r}^{d+j,j}:\CH^{d+j}(X,j)/p^r\r H^{2d+j}_{\text{\'et}}(X,\mathcal{T}_r(j))$. The $\mathcal{T}_r(n)$ are the complexes defined in \cite{Sa07} and called $p$-adic \'etale Tate twists. $\mathcal{T}_r(n)$ is an object in the derived category $D^b(X,\Z/p^r\Z)$ of bounded complexes of \'etale $\Z/p^r\Z$-sheaves on $X$. $\mathcal{T}_r(n)$ is expected to agree with $\Z(n)^\text{\'et}\otimes^\mathbb{L}\Z/p^r\Z$, where $\Z(n)^\text{\'et}$ denotes the conjectural  \'etale motivic complex of Beilinson-Lichtenbaum (see \cite[Sec. 1.3]{SS14}). If $p>n+1$, then $i^*\mathcal{T}_r(n)\cong \mathcal{S}_r(n)$, where $i$ is the inclusion $X_1\r X$ and $\mathcal{S}_r(n)$ is the syntomic complex defined in \cite{Ka87} (see \cite[Sec. 1.4]{Sa07}). In \cite{SS14}, Saito and Sato show the following result on the $p$-adic cycle class map:
\begin{theorem} (\cite[Thm. 1.3.1]{SS14})
Let $X$ be a regular scheme which is proper flat of finite type over the ring of integers $A$ of a $p$-adic local field $K$. Assume that $X$ has good or semistable reduction over $A$ and let $d$ be the fiber dimension of $X$ over $A$. Then the cycle class map
$$\varrho_{p^r}^{d,0}:\CH^d(X)/p^r\r H^{2d}_{\mathrm{\et}}(X,\mathcal{T}_r(d))$$
defined in \cite[Cor. 6.1.4]{Sa07} is surjective. 
\end{theorem}

We will show the following proposition: 
\begin{proposition}\label{propositionrelationmilnorktosyntomic} Let $W(k)$ be the Witt ring of a finite field $k$ of characteristic $p$ and $p>d$. Let $X$ be smooth and projective over $W(k)$. Then for all $j\geq d$ the map
$$"\mathrm{lim}_n" H^{d}(X_1,\mathcal{K}^M_{j,X_n}/p^r)\r H^{d+j}_{\mathrm{\et}}(X_1,\mathcal{S}_r(j)) $$
is an isomorphism of pro-abelian groups. Here we consider $H^{d+j}_{\mathrm{\et}}(X_1,\mathcal{S}_r(j))$ as a constant pro-system.
\end{proposition}

In sum we establish, making use of the above result on the Gersten conjecture, the following commutative diagram analogous to diagram (\ref{diagrammintroduction}) for $X$ smooth over $A=W(k)$ for a finite field $k$ of characteristic $p$, $j\geq 0$ and $p>d+j+1$ (see Proposition/Definition \ref{diagrammintroduction2proof}): 
\begin{equation}\label{diagrammintroduction2}
\begin{xy} 
  \xymatrix{
  \CH^{d+j}(X,j,\Z/p^r\Z) \ar[r]^-{\cong} & H^d(X,\mathcal{K}^M_{d+j,X}/p^r) \ar@{->>}[r]^-{res} \ar@{->>}[d]_{} & "\mathrm{lim}_n" H^d(X_1,\mathcal{K}^M_{d+j,X_n}/p^r) \ar[d]_{\cong} \\
 & H^{2d+j}_{\text{\'et}}(X,\mathcal{T}_r(d+j))  \ar[r]^-{\cong} & H^{2d+j}_{\mathrm{\text{\'et}}}(X_1,\mathcal{S}_r(d+j))   
  }
\end{xy} 
\end{equation}
\begin{notat*}
Unless otherwise specified, all cohomology groups are taken over the Zariski topology.
\end{notat*}
\paragraph{\textit{ Acknowledgement}.} I would like to heartily thank my supervisor Moritz Kerz for fruitful discussions leading to the results of this article.

\section{Parshin chains}\label{sectionparshinchains}
Let $X$ be an excellent scheme.
\begin{definition}
\begin{enumerate}
\item A chain on $X$ is a sequence of points $P=(p_0,...,p_s)$ on $X$ such that 
$$\overline{\{p_0\}}\subset \overline{\{p_1\}}\subset ...\subset \overline{\{p_s\}}.$$
\item A Parshin chain on $X$ is a chain $P=(p_0,...,p_s)$ such that $\mathrm{dim}\overline{\{p_i\}}=i$ for all $0\leq i\leq s$.
\item A $Q$-chain on $X$ is a chain $Q=(p_0,...,p_{s-2},p_s)$ such that $\mathrm{dim}\overline{\{p_i\}}=i$ for $i\in\{0,1,...,s-2,s\}$. We denote by $B(Q)$ the set of all $x\in X$ such that $Q(x)=(p_0,...,p_{s-2},x,p_s)$ is a Parshin chain.
\item Let $Z$ be a closed subscheme of $X$ and $U=X-Z$. A Parshin chain (resp. $\mathcal{Q}$-chain) on $(X,Z)$ is a Parshin chain $P=(p_0,...,p_s)$ (resp. $\mathcal{Q}$-chain $Q=(p_0,...,p_{s-2},p_s)$) such that $p_i\in Z$ for $i\leq s-1$ and $p_s\in U$ (resp. $p_i\in Z$ for $i\leq s-2$ and $p_s\in U$). A Parshin chain (resp. $\mathcal{Q}$-chain) on $(X,X)$ is a Parshin chain in the sense of (2) (resp. $\mathcal{Q}$-chain in the sense of (3)).
\item We say that a Parshin chain $P=(p_0,...,p_s)$ on $X$ is supported on a closed subscheme $Z$ of $X$ if $p_i\in Z$ for all $0\leq i\leq s$.
\item The dimension $d(P)$ of a chain $P=(p_0,...,p_s)$ is defined to be $\text{dim}\overline{\{p_s\}}$.
\end{enumerate}
\end{definition}

\begin{definition}\label{definitionlocalisationparshinchain} 
Let $P=(p_0,...,p_s)$ be a chain on $X$.
\begin{enumerate}
\item We define $\O_{X,P}=\O_{X,p_s}$ and $k(P)=k(p_s)$.
\item We define the finite product of henselian local rings $\O^h_{X,P}$ inductively as follows: If $s=0$, then $\O^h_{X,P}=\O^h_{X,p_0}$. If $s>0$, then assume that the ring $\O^h_{X,P'}$ over $\O_{X,p_0}$ has already been defined for $P'=(p_0,...,p_{s-1})$. Denote $\O^h_{X,P'}$ by $R$. Let $T$ be the finite set of prime ideals of $\O^h_{X,P'}$ lying over $p_s$ and 
$$\O^h_{X,P}=\prod_{\mathfrak{p}\in T}R^h_\mathfrak{p}.$$
Let $k^h(P)$ denote the finite product of residue fields of $\O^h_{X,P}$. 
\end{enumerate}
\end{definition}
We note that $T$ in Definition \ref{definitionlocalisationparshinchain}(2) is finite by \cite[Thm. 18.6.9 (ii)]{EGA4} since $X$ is excellent and therefore in particular noetherian. For a Parshin chain $P$ on $X$ we denote $\text{Spec}\O_{X,P}$ by $X_P$ and $\text{Spec}\O_{X,P}^h$ by $X^h_P$. For more details on Parshin chains see \cite[Sec. 1.6]{KaS86}.

We will need the following facts (see e.g. \cite[Ch. IV.]{Ha66}): Let $X$ be a locally noetherian scheme, $\mathcal{F}$ be a sheaf of abelian groups on $X$ and $\tau\in \{\text{Zar,Nis}\}$. Then there are coniveau spectral sequences 
$$E^{p,q}_1=\bigoplus_{x\in X^{(p)}}H^{p+q}_x(X_{\tau},\mathcal{F})\Rightarrow H^{p+q}(X_{\tau},\mathcal{F})$$
and isomorphisms 
$$H^{q}_x(X_{\Zar},\mathcal{F})\cong H^{q}_x(\mathcal{O}_{X,x},\mathcal{F}) \text{ and } H^{q}_x(X_{\Nis},\mathcal{F})\cong H^{q}_x(\mathcal{O}^h_{X,x},\mathcal{F})$$
for every $x\in X$ and $q\geq 0$. From the coniveau spectral sequence we get complexes
$$...\r \bigoplus_{x\in X^{(p-1)}}H^{p+q-1}_x(X_{\tau},\mathcal{F})\r \bigoplus_{x\in X^{(p)}}H^{p+q}_x(X_{\tau},\mathcal{F})\r \bigoplus_{x\in X^{(p+1)}}H^{p+q+1}_x(X_{\tau},\mathcal{F})\r..$$ 
We denote a morphism $H^{p+q}_y(X_{\tau},\mathcal{F})\r H^{p+q+1}_x(X_{\tau},\mathcal{F})$ arising this way by $\partial_{yx}$. We explain this notation as follows: If $\mathcal{F}=\mathcal{K}^M_{n,X},y\in X^{(p+q)},x\in X^{(p+q+1)},$ and if the Gesten conjecture holds for $\mathcal{K}^M_{n,X}$ (see Section \ref{GCMKp}), then the diagram
$$\begin{xy}
  \xymatrix{
        H^{p+q}_y(X_{\tau},\mathcal{K}^M_{n,X}) \ar[d]^-{\cong} \ar[r]  & H^{p+q+1}_x(X_{\tau},\mathcal{K}^M_{n,X})\ar[d]^-{\cong} \\
   K^M_{n-p-q}(k(y))  \ar[r]^-{} & K^M_{n-p-q-1}(k(x))
  }
\end{xy} $$
commutes, where the lower horizontal map is the tame symbol defined by passing to the normalisation and using the norm map for Milnor K-theory (see f.e. \cite[8.1.1]{GS06}).

Finally recall that the cohomological dimension of $X_{\Zar}$ and $X_{\Nis}$ is at most equal to dim$(X)$.

\section{The Gersten conjecture for Milnor K-theory mod $p$}\label{sectionGCMK}
Let $X$ be an excellent scheme and let $\mathcal{K}^M_{n,X}$ be the improved Milnor K-sheaf defined in \cite{Ke10}. 
\begin{definition}\label{Gcdef} We say that the Gersten conjecture holds for the (Milnor K-)sheaf $\mathcal{K}^M_{n,X}$ if the sequence of sheaves
$$0\r\mathcal{K}^M_{n,X}\rightarrow \bigoplus_{x\in X^{(0)}}i_{x,*}K^M_n(k(x))\rightarrow \bigoplus_{x\in X^{(1)}}i_{x,*}K^M_{n-1}(k(x))\r...$$
is exact.
\end{definition} 
This conjecture is known to hold for $\mathcal{K}^M_{n,X}$ if all local rings of $X$ are regular and equi-characteristic (see \cite{Ke09} and \cite[Prop. 10(8)]{Ke10}). If $X$ is smooth over a henselian local discrete valuation ring of mixed characteristic $(0,p)$, then the Gersten conjecture is not known to hold for the sheaf $\mathcal{K}^M_{n,X}$. However, if $p>n-1$, then we have the following much weaker result which we will use in Section \ref{sectionapplications}.

\begin{proposition}\label{GCMKp}
Let $A$ be a discrete valuation ring with uniformising parameter $\pi$ and residue field $k$ of characteristic $p>0$.
Let $B$ be a local ring, essentially smooth over $A$ with field of fractions $F$ and let $p>(n-1)$. Set $X:=\Spec(B)$. Then the sequence 
$$0\rightarrow K^M_n(B)/p^r\xrightarrow{i_n} K^M_n(F)/p^r\rightarrow \bigoplus_{x\in X^{(1)}} K^M_{n-1}(x)/p^r \rightarrow ...$$
is exact for all $r\geq 1$.
\end{proposition}
\begin{proof}
First note that the Gersten conjecture for Quillen K-theory with finite coefficients holds for $B$ by \cite[Thm. 8.2]{GL00}.

We consider the following commutative diagram:
$$\begin{xy} 
  \xymatrix{
   K^M_n(B)/p^r \ar[r]^{i_n} \ar[d]_{} &  K^M_n(F)/p^r  \ar[d]^{} \ar[r]^{} \ar[d]_{} & \bigoplus_{x\in X^{(1)}} K^M_{n-1}(x)/p^r \ar[d]^{} \\
     K^Q_n(B,\Z/p^r) \ar[r]^{i_n^Q}   \ar[d]^{}  & K^Q_n(F,\Z/p^r) \ar[d]^{} \ar[r]^{}  & \bigoplus_{x\in X^{(1)}} K^Q_{n-1}(x,\Z/p^r) \\
      K^M_n(B)/p^r \ar[r]^{}  &  K^M_n(F)/p^r  \ar[r]^{}  & \bigoplus_{x\in X^{(1)}} K^M_{n-1}(x)/p^r
  }
\end{xy} $$
The composition $\mu:K^M_n(B)\r K^Q_n(B) \r K^M_n(B)$ is multiplication by $(n-1)!$ by \cite[Sec. 4]{NS89} and \cite[Prop. 10(6)]{Ke10}. Let us first show the injectivity of $i_n$: Let $\alpha\in K^M_n(B)/p^r $ and suppose that $i_n(\alpha)=0$. Then $(n-1)!\cdot \alpha=0$ since $i_n^Q$ is injective. For $p>(n-1)$ we have that $(p,(n-1)!)=1$. This implies that $\alpha=0$. The exactness at $K^M_n(F)/p^r$ can be seen as follows: Let $\alpha\in  \text{ker}[K^M_n(F)/p^r\r \bigoplus_{x\in X^{(1)}} K^M_{n-1}(x)/p^r]$. Then $(n-1)!\alpha\in \text{im} (i_n)$ since the square on the upper right commutes (see \cite[p. 449f.]{We13}) and the middle row is exact at $K^Q_n(F,\Z/p^r)$. Again since $(p,(n-1)!)=1$ it follows that $\alpha\in \text{im} (i_n)$.

Exactness at the other places follows for example from \cite[Cor. 4.3]{Ge04}.
\end{proof}

\begin{remark}
See \cite[Cor. 4.4]{NS89} for a similar result.
\end{remark}

We will repeatedly use the following purity statement which follows from the Gersten conjecture:
\begin{lemma}\label{purity}
Let $D$ be an effective Cartier divisor on $X$. Let $i:D\r X$ be the inclusion and $U:=X\setminus \mathrm{supp}D$. Let $\mathcal{K}^M_{n,X|D}$ be the sheaf defined in Definition \ref{definitionrelaitvmilnorksheaf}(2). Let $x\in X$ be a point which is not be contained in $D$ and assume the Gersten conjecture for the sheaf $\mathcal{K}^M_{n,U}$. Then for $t=\mathrm{codim}_X(x)$ there is a canonical isomorphism 
$$H^t_x(X,\mathcal{K}^M_{n,X|D})\cong K^M_{n-t}(k(x)).$$
\end{lemma}
\begin{proof}
First note that $$H^t_x(X,\mathcal{K}^M_{n,X|D})\cong H^t_x(\text{Spec}\O_{X,x},\mathcal{K}^M_{n,X|D}\mid_{\text{Spec}\O_{X,x}})\cong H^t_x(\text{Spec}\O_{X,x},\mathcal{K}^M_{n,\text{Spec}\O_{X,x}}).$$
The purity isomorphism now follows from the Gersten conjecture for $\mathcal{K}^M_{n,\text{Spec}\O_{X,x}}$ which holds by assumption since $x\in U$.
Indeed, applying $\Gamma_x$ (see \cite[p. 225]{Ha66}) to the exact sequence
$$\mathcal{K}^M_{n,\text{Spec}\O_{X,x}}\r...\r \bigoplus_{y\in \O_{X,x}^{(c-1)}}i_{y,*}K^M_{n-t+1}(k(y))\r i_{x,*}K^M_{n-t}(k(x))\r 0$$ 
gives the sequence
$$...\r 0\r K^M_{n-t}(k(x))$$
(the last term is in degree $t$) since $i_{y,*}K^M_{n-t+1}(k(y))$ is the constant sheaf on the integral scheme $\overline{\{y\}}$.
\end{proof}

\section{Some topology on Milnor K-groups}
In this section we define a topology on Milnor K-groups and state two lemmas which we will need in the proof of our main theorem.

Recall that the naive Milnor K-sheaf $\mathcal{K}^{M,\mathrm{naive}}_n$ is defined to be the sheafification of the functor
$$R\mapsto (R^\times)^{\otimes n}/<a_1\otimes ...\otimes a_n|a_i+a_j=1 \text{ for some } i\neq j>$$
from the category of commutative rings to abelian groups and that there is a natural homomorphism of sheaves 
$$\mathcal{K}^{M,\mathrm{naive}}_n\r \mathcal{K}^M_n$$
to the improved Milnor K-sheaf which is surjective (see \cite{Ke10}). In particular there is the following commutative diagram for a commutative local ring $R$, an ideal $I\subset R$ and $K=\text{Frac}(R)$:
$$\begin{xy} 
  \xymatrix{
   \mathcal{K}^{M}_n(R) \ar[rr]^{} \ar[rd] &    & \mathcal{K}^{M}_n(R/I) \\
       & K^M_n(K)  &  \\
      \mathcal{K}^{M,\mathrm{naive}}_n(R) \ar@{->>}[uu]^{} \ar[ru]^{}\ar[rr]^{}  &    & \mathcal{K}^{M,\mathrm{naive}}_n(R/I) \ar@{->>}[uu]^{}
  }
\end{xy} $$
This implies that that when defining a topology on $K^M_n(K)$ as in the following Definition \ref{definitionrelaitvmilnorksheaf}(4) we may work with both $\mathcal{K}^{M,\mathrm{naive}}_n$ or $\mathcal{K}^M_n$. We will use the improved Milnor K-sheaf and at some points implicitly use its generation by symbols.
\begin{definition}\label{definitionrelaitvmilnorksheaf}
\begin{enumerate}
\item For a commutative ring $R$ and an ideal $I\subset R$ we define $K^{M}_n(R,I)$ to be $\mathrm{ker}[\mathcal{K}^{M}_n(R)\r \mathcal{K}^{M}_n(R/I)]$ and similarly for $K^{M,\mathrm{naive}}_n(R,I)$.
\item Let $D$ be an effective Cartier divisor on $X$. We define $\mathcal{K}^{M}_{n,X\mid D}$ to be the kernel of the restriction map $\mathcal{K}^{M}_{n,X} \r i_*\mathcal{K}^{M}_{n,D}$ for $i:D\r X$ the inclusion. Again similarly for $K^{M,\mathrm{naive}}_{n,X\mid D}$.
\item Let $R$ be an excellent semi-local integral domain of dimension $1$ with field of fractions $K$. We endow $R$ with the $J_R$-adic topology, where $J_R$ is the Jacobsen radical of $R$. We endow $K^M_n(K)$ with the structure of a topological group by taking the subgroups generated by $\{U_1,...,U_n\}$, where $U_i$ ranges over all open subgroups of $R^\times$, as a fundamental system of neighbourhoods of $0$ in $K^M_n(K)$. 
\item For a Parshin chain $P=(p_0,...,p_{s-1},p_s)$, and $P'=(p_0,...,p_{s-1})$, on an excellent scheme $X$ and $Y=\overline{\{p_s\}}$ we define a topology on $K^{M}_n(k(P))$ (resp. $K^{M}_n(k^h(P))$) by taking the images of $K^{M}_n(\O_{Y,P'},I)$ (resp. $K^{M}_n(\O_{Y,P'}^h,I)$) as a fundamental system of neighbourhoods of $0$, where $I$ ranges over all open ideals, with respect to the topology defined in (3), of the one dimensional local ring $\O_{Y,P'}$ (resp. semi-local ring $\O_{Y,P'}^h$).
\end{enumerate}
\end{definition}

\begin{remark}
If we set $R:=\O_{Y,P'}$ (resp. $\O_{Y,P'}^h$), then the topologies on $K^M_n(K), K=\text{Frac}(R),$ defined in (3) and (4) coincide.
\end{remark}
 
\begin{example}
Let $m\geq 0$ be an integer. If $R$ in (3) is a discrete valuation ring with quotient field $K$, maximal ideal $\mathfrak{p}\subset R$ and generic point $\eta$, then the subgroups generated by $K^M_n(K,m):=\{1+\mathfrak{p}^m,R^{\times},...,R^{\times}\}$ of $K^M_n(K)$ generate the topology on $K^M_n(K)$ with respect to the Parshin chain $(\mathfrak{p},\eta)$.
\end{example}

\begin{lemma}\label{lemmaintegralclosure}(Cf. \cite[Prop. 2]{Ka83}) Let $R$ be an excellent semi-local integral domain of dimension $1$ with field of fractions $K$.
Let $\tilde{R}$ be the integral closure of $R$ in $K$. Then the topology of $K^M_n(K)$ defined by $\tilde{R}$ coincides with that defined by $R$.
\end{lemma}
\begin{proof}
Since $R$ is excellent, the normalisation is finite and there is some $f\in J_R\setminus \{0\}$ such that $f\tilde{R}\subset R$. Therefore for every $i\geq 1$
$$1+f^{i+1}\tilde{A}\subset 1+f^iA.$$
\end{proof}

\begin{lemma}\label{continuous}(cf. \cite[Prop. 2.7]{KaS86}, \cite[Lem. 6.2]{Ke11}) Let $X$ be an excellent integral scheme. Let $U$ be a regular open subscheme of $X$ and $D$ an effective Weil divisor with support $X-U$. Let $y\in U$ and $x$ be of codimension $1$ on $\overline{\{y\}}$. Let dim$\O_{X,y}=t$ and assume the Gersten conjecture for the sheaf $\mathcal{K}^M_{n,U}$. Then the map 
$$\partial_{yx}:K^M_{n-t}(k(y))\cong H^t_{y}(X,\mathcal{K}^M_{n,X|D})\r H^{t+1}_x(X,\mathcal{K}^M_{n,X|D})$$
annihilates the image of $K^M_{n-t}(\O_{Y,x},J_x)$ for some non-zero ideal $J_x\subset \O_{Y,x}$.
In particular the kernel of $\partial_{yx}$ is open with respect to the topology defined in Definition \ref{definitionrelaitvmilnorksheaf}(4) and the Parshin chain $(x,y)$.
\end{lemma}
\begin{proof}
We proceed by induction on $t$. The case $t=0$ is clear since in that case $Y=X$ and $H^{1}_x(X,\mathcal{K}^M_{n,X|D})\cong K^M_n(k(y))/K^M_n(\O_{Y,x},J)$ for $x\in D^{(0)}$ and $J$ corresponding to $D$.

If $t\geq 1$, then we take some point $z\in X^{t-1}$ such that $y$ lies in the regular locus of $\overline{\{z\}}$. Consider the complex
\begin{multline}\label{complex} H^{t-1}_z(X,\mathcal{K}^M_{n,X|D})\r 
\bigoplus_{y'\in \text{Spec}\O_{Z,x}^{(1)}\setminus D} H^t_{y'}(X,\mathcal{K}^M_{n,X|D})\oplus \bigoplus_{y''\in \text{Spec}\O_{Z,x}^{(1)}\cap D} H^t_{y''}(X,\mathcal{K}^M_{n,X|D})\\ \r H^{t+1}_x(X,\mathcal{K}^M_{n,X|D})
\end{multline}
coming from the coniveau spectral sequence in Section \ref{sectionparshinchains}.
Applying the induction assumption to $H^{t-1}_z(X,\mathcal{K}^M_{n,X|D})\r  H^t_{y''}(X,\mathcal{K}^M_{n,X|D})$ for all $y''\in \text{Spec}\O_{Z_x}^{(1)}\cap D$, we see that it suffices to show that the map
$$K^M_{n-t+1}(k(z))\xrightarrow{(\partial,\text{Id})}\bigoplus_{y'\in \text{Spec}\O_{Z,x}^{(1)}\setminus D}K^M_{n-t}(k(y')) \oplus \bigoplus_{y''\in \text{Spec}\O_{Z,x}^{(1)}\cap D}K^M_{n-t+1}(K(z))/K^M_{n-t+1}(\O_{Z,y''},J_{y''})$$
annihilates the image of $K^M_{n-t}(\O_{Y,x},J_x)$ in $K^M_{n-t}(k(y))$ for some non-zero ideal $J_x\subset \O_{Y,x}$ given some non-zero ideals $J_{y''}\subset \O_{Z,y''}$. Indeed, in that case if $\alpha\in\text{Im}(K^M_{n-t}(\O_{Y,x},J_x)\r K^M_{n-t}(k(y)))$, then there is some $\beta\in K^M_{n-t+1}(k(z))$ such that $\partial_{zy}(\beta)=\alpha$ and such that $(\bigoplus_{y'\neq y\in \text{Spec}\O_{Z_x}^{(1)}}\partial_{zy''},\text{Id})(\beta)=0$. Since (\ref{complex}) is a complex, this implies that $\partial_{yx}(\alpha)=0$.

Now let $A:=\O_{Z,x}$. By Lemma \ref{lemmaintegralclosure} we may assume that the $\O_{Z,y''}$ are normal (semi-local) rings. By the definition of $\partial$ we may work with the normalisation $\tilde{A}
$ of $A$. Let $\{y''_1,...,y_r''\}=\text{Spec}\O_{Z,x}^{(1)}\cap D$ and let $J^{(y_i'')}$ be ideals in $\tilde{A}$ such that $J^{(y_i'')}\O_{Z,y_i''}=J_{y_i''}$. Let $\mathfrak{q}$ be the prime ideal corresponding to $y$.
Let $\pi\in A$ such that $v_{\mathfrak{q}}(\pi)=1$. Let $\{\mathfrak{p}_{1+r},...,\mathfrak{p}_{t}\}$ be the finite set of prime ideals in $\tilde{A}$ such that $v_{\mathfrak{p}_i}(\pi)>0,1+r\leq i\leq t$. By a standard approximation lemma (see e.g. \cite[Lem. 9.1.9(b)]{Li02}) we can choose an element $\pi_i$ for all $i$ with $r+1\leq i\leq t$ satisfying $v_{\mathfrak{p}_i}(\pi_i)=1$ and $v_{\mathfrak{q}}(\pi_i)=0$. 
Now we can choose a non-zero ideal 
\begin{equation}\label{surjectionlifting}
J^{(x)}\subset J^{(y_1'')}...J^{(y_r'')}(\pi_{r+1})...(\pi_{t})(\tilde{A}/\mathfrak{q})\twoheadleftarrow J^{(y_1'')}...J^{(y_r'')}(\pi_{r+1})...(\pi_{t})\tilde{A}.
\end{equation}
Let $J_x:=J^{(x)}\O_{Y,x}$. Now given a symbol $\alpha:=\{\bar{a}_1,...,\bar{a}_{n-t}\}\in K^M_{n-t}(\O_{Y,x},J_x)$ with $\bar{a}_1\in 1+J_x$, lift $\alpha$ to $\beta:=\{\pi,a_1,...,a_{n-t}\}\in K^M_{n-t+1}(k(z))$ lifting $\bar{a}_1$ via the surjection in (\ref{surjectionlifting}) to $a_1$ and lifting the other $\bar{a}_i$ arbitrarily. Then $\beta$ satisfies the required properties since
\begin{enumerate}
\item $\partial_{zy}(\beta)=\alpha$.
\item If $\tilde{y}'\notin \text{div}(\pi)$, then $\partial_{z\tilde{y}'}=0$ since $\pi,a_1,...,a_{n-t}\in \O_{\tilde{A},\tilde{y}'}^\times$.
\item If $\tilde{y}'\in \text{div}(\pi)$, i.e. $\tilde{y}'\sim \mathfrak{p}_i$, then $\partial_{z\tilde{y}'}=0$ since $a_1=1$ mod $(\pi_i)$.
\item $a_1\in K^M_{n-t+1}(\O_{Z,y''},J_{y''})$ for all $y''\in \text{Spec}\O_{Z,x}^{(1)}\cap D$.
\end{enumerate}
\end{proof}

\begin{remark}
In \cite[Prop. 2.7]{KaS86} the above lemma was proved in the Nisnevich topology. The proof in the Zariski topology follows the argument in $\textit{loc. cit.}$ We recall the proof for the convenience of the reader and to convince them of this claim. In \cite[Lem. 6.2]{Ke11} the last step of the proof is given under the assumption that $A$ is a two-dimensional excellent henselian local ring. 
\end{remark} 

\begin{lemma}\label{dense}
Given a family of inequivalent discrete valuations $v_1,...,v_s$ on a valued field $F$, the diagonal map
$$K^M_n(F)\r \bigoplus_{v_i}K^M_n(F_{v_i}),$$
has dense image, were we write $F_{v_i}$ instead of $F$ in order to indicate which valuation defines the topology on $F$.
\end{lemma}
\begin{proof}
This follows from standard approximation theorems for $F$. See e.g. \cite[II.3.4]{Ne92}.
\end{proof}

\section{Main theorem}\label{sectionmaintheorem}
In this section we prove Theorem \ref{maintheoremalgebraization}. 

We return to the situation of the introduction. Let $A$ be an excellent henselian discrete valuation ring with uniformising parameter $\pi$ and residue field $k$ and let $X$ be a smooth projective scheme over Spec$(A)$ of relative dimension $d$. Let $X_n:=X\times_A A/(\pi^n)$, i.e. $X_1$ is the special fiber and the $X_n$ are the respective thickenings of $X_1$. 
\begin{proposition}\label{vanishing}
For all $j\geq 0$ the group 
$$H^{d+1}_\mathrm{Zar}(X,\mathcal{K}^M_{j+d,X\mid X_n})=0.$$
\end{proposition}
\begin{proof} 
By the coniveau spectral sequence and cohomological vanishing we have to show that the map
$$\bigoplus_{y\in X^{(d)}}H_y^{d}(X,\mathcal{K}^M_{d+j,X\mid {X_n}})\rightarrow \bigoplus_{x\in X^{(d+1)}}H^{d+1}_x(X,\mathcal{K}^M_{d+j,X\mid {X_n}})$$
is surjective. In order to show this, we show that the map
$$\bigoplus_{y\in (\text{Spec}\O_{X,x}[\frac{1}{\pi}])^{d}}H_y^{d}(X,\mathcal{K}^M_{d+j,X\mid {X_n}})\rightarrow H^{d+1}_x(X,\mathcal{K}^M_{d+j,X\mid {X_n}})$$
is surjective for any $x\in X^{(d+1)}$. This suffices since $(\text{Spec}\O_{X,x}[\frac{1}{\pi}])^d\subset X^{(d)}$ and since, as $A$ is henselian, any $y\in (\text{Spec}\O_{X,x}[\frac{1}{\pi}])^{d}$ restricts to just one closed point $x\in X^{(d+1)}$.
Let us start with the case $d=0$: Let $X_x':=X_x-x$. Then
$$H^{1}_x(X,\mathcal{K}^M_{j,X\mid X_n})\cong H^0(X_x',\mathcal{K}^M_{j,X\mid {X_n}})/H^0(X_x,\mathcal{K}^M_{j,X\mid X_n}),$$
and $H_{\mu}^{0}(X,\mathcal{K}^M_{j,X\mid {X_n}})$, $\mu$ being the generic point of $X$, surjects onto $H^{1}_x(X,\mathcal{K}^M_{j,X\mid X_n})$ since $H_{\mu}^{0}(X,\mathcal{K}^M_{j,X\mid {X_n}})$ is isomorphic to $H^0(X_x',\mathcal{K}^M_{j,X\mid {X_n}})$.

Let $d\geq 1$ and $x\in X^{(d+1)}$. We have that 
$$H^{d+1}_x(X,\mathcal{K}^M_{d+j,X\mid {X_n}})\cong H^{d}(X_{x}',\mathcal{K}^M_{d+j,X\mid {X_n}})$$ 
and again it follows from the coniveau spectral sequence and cohomological vanishing that $ H^{d}(X_x',\mathcal{K}^M_{d+j,X\mid {X_n}})$ is isomorphic to
$$ \text{coker} (\bigoplus_{z\in (X_{x})^{d-1}}H^{d-1}_z(X_{x},\mathcal{K}^M_{d+j,X\mid {X_n}}) \r \bigoplus_{y\in (X_{x})^d}H^d_y(X_{x},\mathcal{K}^M_{d+j,X\mid {X_n}})). $$
By Lemma \ref{purity} we have that 
$$H^d_y(X_{x},\mathcal{K}^M_{d+j,X\mid {X_n}})\cong \mathcal{K}^M_{j,X}(k(x,y))$$
for $y\in (X_{x}[\frac{1}{\pi}])^d$. It therefore suffices to move elements of $H^d_y(X_{x},\mathcal{K}^M_{d+j,X\mid {X_n}})$ for $y\in X_x^{d-r}\setminus (X_{x}[\frac{1}{\pi}])^d$ to the horizontal components, i.e. with $y\in (X_{x}[\frac{1}{\pi}])^d$, using the 'Q-chains' $H^{d-1}_z(X_{x},\mathcal{K}^M_{d+j,X\mid {X_n}})$. 

We write $P_r$ for a Parshin chain $(x,...)$ of dimension $r$ and let $x_{P_r}$ denote the closed point of $X_{P_r}$ and $X_{P_r}'$ the open subscheme $X_{P_r}\setminus \{x_{P_r}\}$. We proceed by descending in induction in $r\geq 0$, starting with $r=d$, to show that the map  
$$\bigoplus_{y\in (X_{P_r}[\frac{1}{\pi}])^{d-r}}H^{d-r}_y(X_{P_r},\mathcal{K}^M_{d+j,X\mid {X_n}})\r H^{d-r+1}_{x_{P_r}}(X_{P_r},\mathcal{K}^M_{d+j,X\mid {X_n}})$$
is surjective for all Parshin chains $P_r$ supported on $X_1$. 

The group $H^{d-r+1}_{x_{P_r}}(X_{P_r},\mathcal{K}^M_{d+j,X\mid {X_n}})$ is isomorphic to
\begin{equation*}
\begin{split}
\text{coker} (\bigoplus_{z\in (X_{P_r})^{d-1-r}}H^{d-1-r}_z(X_{P_r},\mathcal{K}^M_{d+j,X\mid {X_n}}) \r \bigoplus_{y\in (X_{P_r})^{d-r}}H^{d-r}_y(X_{P_r},\mathcal{K}^M_{d+j,X\mid {X_n}})) \\ \cong H^{d-r}(X_{P_r}',\mathcal{K}^M_{d+j,X\mid {X_n}})
\end{split}
\end{equation*} 
for $r<d$ and to 
$$H^{0}(X_{P_d}',\mathcal{K}^M_{d+j,X\mid {X_n}})/H^{0}(X_{P_d},\mathcal{K}^M_{d+j,X\mid {X_n}})$$ 
for $r=d$.
If $r=d$, then $H^{0}_y(X_{P_d},\mathcal{K}^M_{d+j,X\mid {X_n}})$ is isomorphic to $H^{0}(X_{P_d}',\mathcal{K}^M_{d+j,X\mid {X_n}})$ which implies the induction beginning.

We now do the induction step. Let $\alpha\in H^{d-r}_y(X_{P_r},\mathcal{K}^M_{d+j,X\mid {X_n}})$ for $y\in X_{P_r}^{d-r}\setminus (X_{P_r}[\frac{1}{\pi}])^{d-r}$. Then the map
$$\bigoplus_{z\in (X_{(P_r,y)})^{d-r-1}}H^{d-r-1}_z(X_{(P_r,y)},\mathcal{K}^M_{d+j,X\mid {X_n}})\r H^{d-r}_y(X_{P_r},\mathcal{K}^M_{d+j,X\mid {X_n}})$$
is surjective. This follows again from the coniveau spectral sequence and the isomorphism $H^{d-r}_y(X_{P_r},\mathcal{K}^M_{d+j,X\mid {X_n}})\cong H^{d-r-1}(X_{(P_r,y)}',\mathcal{K}^M_{d+j,X\mid {X_n}})$. By assumption we have that 
$$\text{coker} (\bigoplus_{t\in (X_{(P_r,y)})^{d-r-2}}H^{d-r-2}_t(X_{(P_r,y)},\mathcal{K}^M_{d+j,X\mid {X_n}}) \r \bigoplus_{z\in (X_{(P_r,y)})^{d-r-1}}H^{d-r-1}_z(X_{(P_r,y)},\mathcal{K}^M_{d+j,X\mid {X_n}}))$$
is generated by $K^M_{r+j+1}(k(P))$ for all Parshin chains $P=(P_r,y,z)$ of dimension $r+2$ on $(X,X_1)$. We may assume that $\alpha$ is in the image of $K^M_{r+j+1}(k(P))$ for some such $P$. Then by Lemma \ref{continuous} the kernel of the map $$\partial_{zy}-\alpha:K^M_{r+j+1}(k(P))\r H^{d-r}_y(X_{P_r},\mathcal{K}^M_{d+j,X\mid {X_n}})$$
is open in $K^M_{r+j+1}(k(P))$ and the kernel of the map $$\partial_{zy'}:K^M_{r+j+1}(k(P))\r H^{d-r}_{y'}(X_{P_r},\mathcal{K}^M_{d+j,X\mid {X_n}})$$
is open in $K^M_{r+j+1}(k(P_r,y',z))$ for all $y'\neq y\in X_{P_r}^{d-r}\cap X_1$ with $\overline{\{x_{P_r}\}}\subset \overline{\{y'\}}\subset \overline{\{z\}}$. By Lemma \ref{dense} and Lemma \ref{lemmaintegralclosure} the diagonal image of $$K^M_{r+j+1}(k(Q))\cong H^{d-r-1}_z(X_{P_r},\mathcal{K}^M_{d+j,X\mid {X_n}})$$ for a Q-chain $Q=(P_r,z)$ is dense in the direct sum (with finitely many summands) $K^M_{r+j+1}(k(P))\oplus_{y'\neq y\in X_{P_r}^{d-r}\cap  \overline{\{z\}}_1}K^M_{r+j+1}(k(P_r,y',z))$ which implies that $\alpha$ is in the image of some $\beta\in H^{d-r-1}_z(X_{x},\mathcal{K}^M_{d+j,X\mid {X_n}})$ with $\beta$ mapping to zero in $H^{d-r}_{y'}(X_{x},\mathcal{K}^M_{d+j,X\mid {X_n}})$ for all $y'\neq y\in((X_{x})_{1})^d$.
\end{proof}

\begin{remark}
The proof of Proposition \ref{vanishing} is inspired by the proof of Theorem 2.5 in \cite{KaS86} and the proof of Theorem 8.2 in \cite{Ke11}. In both of these articles the authors work in the Nisnevich topology. We note that the proof of Proposition \ref{vanishing} also works in the Nisnevich topology if for every Parshin chain $P=(p_0,...,p_s)$ we replace $\O_{X,P}=\O_{X,p_s}$ by $\O_{X,P}^h$ according to Definition \ref{definitionlocalisationparshinchain}.
We therefore get that  
$$H^{d+1}_\mathrm{Nis}(X,\mathcal{K}^M_{j+d,X\mid X_n})=0$$
for all $j\geq 0$. 
\end{remark}

\begin{remark}\label{zarnismotivicthickenings}
If $\mathrm{ch}(k)=0$ and $A=k[[t]]$ or if $A$ is the Witt ring $W(k)$ of a perfect field $k$ of ch$(k)>2$, then there are exact sequences of sheaves
$$0\r \Omega^{r-1}_{X_1}\rightarrow \mathcal{K}^M_{r,X_n}\rightarrow \mathcal{K}^M_{r,X_{n-1}}\rightarrow 0$$
and
$$0\r \Omega^{r-1}_{X_1}/B_{n-2}\Omega^{r-1}_{X_1}\rightarrow \mathcal{K}^M_{r,X_n}\rightarrow \mathcal{K}^M_{r,X_{n-1}}\rightarrow 0$$
respectively, by \cite[Sec. 2]{BEK14'} and \cite[Sec. 12]{BEK14}. Under the above assumptions this implies that the canonical map
$$H^{d}_{\mathrm{Zar}}(X_1,\mathcal{K}^M_{d+j,X_n})\r H^{d}_\mathrm{Nis}(X_1,\mathcal{K}^M_{d+j,X_n})$$ is an isomorphism for all $n\in \N_{>0}$. Indeed, it follows from the Gersten conjecture for the Milnor K-sheaf $\mathcal{K}^M_{*,X_1}$ that the maps $H^{i}_{\mathrm{Zar}}(X_1,\mathcal{K}^M_{d+j,X_1})\r H^{i}_\mathrm{Nis}(X_1,\mathcal{K}^M_{d+j,X_1})$ are isomorphisms for all $i$ and the sheaves $\Omega^{r-1}_{X_1}$ and $\Omega^{r-1}_{X_1}/B_{n-2}\Omega^{r-1}_{X_1}$ are coherent. The claim now follows by induction on $n$. \end{remark}

\begin{corollary}\label{corollaryvanishingidelicH}
\begin{enumerate}
\item The restriction map $res:H^{d}(X,\mathcal{K}^M_{d+j,X})\xrightarrow{} H^{d}(X_1,\mathcal{K}^M_{d+j,X_n})$ is surjective. In particular the map of pro-systems 
$$res: H^{d}(X,\mathcal{K}^M_{d+j,X}) \r "\mathrm{lim}_n" H^{d}(X_1,\mathcal{K}^M_{d+j,X_n})$$
is an epimorphism in $\text{pro-}\text{Ab}$ for all $j\geq 0$.
\item The restriction map $res:H^{d}(X,\mathcal{K}^M_{d+j,X}/p^r)\xrightarrow{} H^{d}(X_1,\mathcal{K}^M_{d+j,X_n}/p^r)$ is surjective. In particular the map of pro-systems 
$$res: H^{d}(X,\mathcal{K}^M_{d+j,X}/p^r) \r "\mathrm{lim}_n" H^{d}(X_1,\mathcal{K}^M_{d+j,X_n}/p^r)$$
is an epimorphism in $\text{pro-}\text{Ab}$ for all $j\geq 0$. 
\end{enumerate}
Here and in the following we always consider $H^{d}(X,\mathcal{K}^M_{d+j,X})$ (resp. $H^{d}(X,\mathcal{K}^M_{d+j,X}/p^r)$) as a constant pro-system in in $\text{pro-}\text{Ab}$.
\end{corollary}
\begin{proof}
For (1) consider the short exact sequence
$$0\r \mathcal{K}^M_{d+j,X\mid X_n} \r \mathcal{K}^M_{d+j,X} \r \mathcal{K}^M_{d+j,X_n} \r 0$$
 and the induced long exact sequence 
$$...\r H^{d}(X,\mathcal{K}^M_{d+j,X\mid X_n})\xrightarrow{i} H^{d}(X,\mathcal{K}^M_{d+j,X})\xrightarrow{res} H^{d}(X_1,\mathcal{K}^M_{d+j,X_n})\r H^{d+1}(X,\mathcal{K}^M_{d+j,X\mid X_n})\r...$$
The statement now follows from Proposition \ref{vanishing} and the fact that $"\mathrm{lim}_n"$ is exact when considered as a functor $\text{Hom}(I^{\text{op}},\text{Ab})\r \text{pro-}\text{Ab}$, where $I$ is a small filtering category (see \cite[App. Prop. 4.1]{AM69}). 

(2) can be seen as follows: Since $\otimes \Z/p^r\Z$ is right exact, there is a short exact sequence 
$$0\r \mathcal{K}^M_{d+j,X\mid X_n}/p^r/\mathcal{I} \r \mathcal{K}^M_{d+j,X}/p^r \r \mathcal{K}^M_{d+j,X_n}/p^r \r 0$$
for some sheaf of abelian groups $\mathcal{I}$. This induces an exact sequence
$$H^{d+1}(X,\mathcal{K}^M_{d+j,X\mid X_n}/p^r)\r H^{d+1}(X,\mathcal{K}^M_{d+j,X\mid X_n}/p^r/\mathcal{I}) \r H^{d+2}(X,\mathcal{I}).$$
By \cite[Thm. 3.6.5]{Gr57} the group $H^{d+2}(X,\mathcal{I})$ vanishes for dimensional reasons. The group $H^{d+1}(X,\mathcal{K}^M_{d+j,X\mid X_n}/p^r)$ vanishes by the same arguments as in the integral case or in fact from the surjectivity of the map $H^{d+1}(X,\mathcal{K}^M_{d+j,X\mid X_n})\r H^{d+1}(X,\mathcal{K}^M_{d+j,X\mid X_n}/p^r)$ which also holds for dimensional reasons. Together this implies the vanishing of $H^{d+1}(X,\mathcal{K}^M_{d+j,X\mid X_n}/p^r/\mathcal{I})$. This implies the statement by the same argument as in the proof of (1).
\end{proof}

\begin{corollary}\label{corollaryvanishingidelicH2} 
If $A$ is equi-characteristic, then the map 
$$res: \CH^{d+j}(X,j) \r "\mathrm{lim}_n" H^{d}(X_1,\mathcal{K}^M_{d+j,X_n})$$
is an epimorphism in $\text{pro-}\text{Ab}$ for all $j\geq 0$. If $A$ is of mixed characteristic $(0,p)$ with $p>d+j-1$, then the map 
$$res: \CH^{d+j}(X,j, \Z/p^r) \r "\mathrm{lim}_n" H^{d}(X_1,\mathcal{K}^M_{d+j,X_n}/p^r)$$
is an epimorphism in $\text{pro-}\text{Ab}$ for all $j\geq 0$.
\end{corollary} 
\begin{proof}
For $j=0$, Corollary \ref{corollaryvanishingidelicH} implies the first assertion since the Gersten conjecture for the sheaf $\mathcal{K}^M_{n,X}$ holds for regular schemes of equal characteristic and the second assertion since the Gersten conjecture holds for $\mathcal{K}^M_{n,X}/p^r$ if $p>n-1$ by Proposition \ref{GCMKp}.

If $j>0$, then the identifications of $\CH^{d+j}(X,j)$ with $H^{d}(X,\mathcal{K}^M_{d+j,X})$ in the equi-characteristic case and $\CH^{d+j}(X,j,\Z/p^r\Z)$ with $H^{d}(X,\mathcal{K}^M_{d+j,X}/p^r)$ in the mixed characteristic case require in addition to the above mentioned results on the Gersten conjecture the Gersten conjecture for higher Chow groups (see \cite[Sec. 2]{Lu17}). This holds if $A$ is equi-characteristic by \cite[Sec. 10]{Bl86} and the method developed by Panin in \cite{Pa03} to extend the Gersten conjecture to the equi-dimensional setting. In mixed characteristic the Gersten conjecture for higher Chow groups with $\Z/p^r\Z$-coefficients holds by \cite[Cor. 4.3]{Ge04}.
\end{proof}

\begin{remark}\label{remarkoninjectivity}
Consider again the short exact sequence
$$0\r \mathcal{K}^M_{d,X\mid X_n} \r \mathcal{K}^M_{d,X} \r \mathcal{K}^M_{d,X_n} \r 0$$
 and the induced long exact sequence 
$$...\r H^{d}(X,\mathcal{K}^M_{d,X\mid X_n})\xrightarrow{i} H^{d}(X,\mathcal{K}^M_{d,X})\xrightarrow{res} H^{d}(X_1,\mathcal{K}^M_{d,X_n})\xrightarrow{0} H^{d+1}(X,\mathcal{K}^M_{d,X\mid X_n})\r...$$
We denote the image of $H^{d}(X,\mathcal{K}^M_{d,X\mid X_n})$ under $i$ by $F_n^X$.

As mentioned in the introduction, Kerz, Esnault and Wittenberg conjecture in \cite[Sec. 10]{KEW16} that if $\text{ch}(\text{Quot}(A))=0$ and $k$ is perfect of characteristic $p>0$ and if we assume that the Gersten conjecture for $\mathcal{K}^M_X$ holds, then the map 
$$res:\CH_1(X)/p^r\rightarrow "\mathrm{lim}_n"H^d(X_n,\mathcal{K}^M_{X_n,d}/p^r)$$ is an isomorphism in pro-Ab.
We note that this conjecture would be implied by the following conjecture:
\begin{conj} $$F_n^X=<g_*F_n^Y|g:Y\r X \text{ projective}, Y/A \text{ smooth projective relative curve}>.$$ \end{conj}
This can be seen as follows: By definition, $F_n^Y\subset H^{1}(Y,\mathcal{K}^M_{1,Y})$ is the image of $H^{1}(Y,\mathcal{K}^M_{1,Y|Y_n})=H^{1}(Y,\mathcal{O}^\times_{Y|Y_n})$ under $i$. By the theorem on formal functions and the $p$-adic logarithm isomorphism, assuming that $p$ is large enough, $H^{1}(Y,\mathcal{O}^\times_{Y|Y_n})\cong H^{1}(Y,p^n\mathcal{O}_Y)$ and therefore the composition
$$H^{1}(Y,\mathcal{O}^\times_{Y|Y_{n+1}})\r H^{1}(Y,\mathcal{O}^\times_{Y|Y_n})$$
is multiplication by $p$.
This implies that $"\mathrm{lim}_n"F_n^Y\otimes \mathbb{Z}/p^r=0$ and therefore that $"\mathrm{lim}_n"F_n^X\otimes \mathbb{Z}/p^r=0$.
\end{remark}

\begin{corollary}\label{conjrel1} Let $k$ be a finite field of characteristic $p>2$ and $A=W(k)$ the Witt ring of $k$. Let $X$ be a smooth projective scheme of relative dimension $1$ over $A$. Then the map $$res:\CH^1(X)/p^r\rightarrow "\mathrm{lim}" H^1(X_1,\mathcal{K}^M_{1,X_n}/p^r)$$
is an isomorphism in the category of pro-systems of abelian groups.
\end{corollary}
\begin{proof}
The injectivity follows from the arguments in Remark \ref{remarkoninjectivity} assuming $p>2$ for the $p$-adic logarithm isomorphism. The surjectivity follows from Corollary \ref{corollaryvanishingidelicH}. 
\end{proof}

\section{Relation with the $p$-adic cycle class map}\label{sectionapplications}
In this section we prove Proposition \ref{propositionrelationmilnorktosyntomic}.

Let $k$ be a finite field of ch$(k)=p>0$, $A=W(k)$ and $X$ be a smooth projective scheme over $A$ of fiber dimension $d$. We let $X_1/k$ denote the reduced special fiber. Let $\tau\in\{\text{Nis, \'et}\}$ and $X_{1,\tau}$ be the respective small site. Let $\epsilon:X_{1,\text{\'et}}\r X_{1,\text{Nis}}$ be the canonical map of sites. 
\begin{definition}(\cite[Def. A.3]{BEK14})  
\begin{enumerate}
\item[(a)] By  $\mathrm{Sh}(X_{1,\tau})$ we denote the category of sheaves of abelian groups on $X_{1,\tau}$. By $\mathrm{C}(X_{1,\tau})$ we denote the category of unbounded complexes in $\mathrm{Sh}(X_{1,\tau})$.
\item[(b)] By $\mathrm{Sh}_{\mathrm{pro}}(X_{1,\tau})$ we denote the category of pro-sytems in $\mathrm{Sh}(X_{1,\tau})$.
\item[(c)] By $\mathrm{C}_{\mathrm{pro}}(X_{1,\tau})$ we denote the category of pro-systems in $\mathrm{C}(X_{1,\tau})$.
\item[(d)] By $\mathrm{D}_{\mathrm{pro}}(X_{1,\tau})$ we denote the Verdier localization of the homotopy category of $\mathrm{C}_{\mathrm{pro}}(X_{1,\tau})$, where we kill objects which are represented by systems of complexes which have level-wise vanishing cohomology sheaves.
\end{enumerate}
\end{definition}

\begin{definition}
We define 
$$W_\cdot \Omega^\bullet_{X_1}\in \mathrm{C}_{\mathrm{pro}}(X_1)_\tau$$
to be the pro-system of de Rham-Witt complexes in the \'etale or Nisnevich topology (see \cite{Il79}).
We define 
$$W_\cdot \Omega^r_{X_1,\mathrm{log}}\in \mathrm{Sh}_{\mathrm{pro}}(X_1)_\tau$$
to be the pro-system of \'etale or Nisnevich subsheaves in $W_r\Omega^j_{X_1}$ which are locally generated by symbols
$$d\log\{[a_1]\}\cdot...\cdot d\log\{[a_j]\}$$
with $a_1,...,a_j\in\O_{X_1}^\times$ local sections and where $[-]$ is the Teichm\"uller lift (see \cite[p. 505, (1.1.7)]{Il79}).
\end{definition}

\begin{definition}
Assuming $j<p$, we define $\mathcal{S}_{r}(j)_{\et}$ to be the syntomic complex defined in \cite[Def. 1.6]{Ka87}. 
We denote the corresponding object in $\mathrm{D}_{\mathrm{pro}}(X_1)_{\et}$ by $\mathcal{S}_{X_.}(j)_{\et}$.
\end{definition}

\begin{definition}(\cite[Sec. 4]{BEK14})
We define $\mathcal{S}_{r}(j)_{\mathrm{Nis}}:=\tau_{\leq j}R\epsilon_*\mathcal{S}_{r}(j)_{\et}$ and $\mathcal{S}_{X_.}(j)_{\mathrm{Nis}}:=\tau_{\leq j}R\epsilon_*\mathcal{S}_{X_.}(j)_{\et}$, where $\tau_{\leq j}$ is the good truncation.
\end{definition}

Let $j<p$. In \cite[Sec. 7]{BEK14}, Bloch, Esnault and Kerz define a motivic pro-complex 
$$\Z_{X_.}(j):=\text{cone}(\mathcal{S}_{X_.}(j)\oplus \Z_{X_1}(j)\r W_.\Omega^j_{X_1,\mathrm{log}}[-j])[-1]$$
in the Nisnevich topology. $\Z_{X_.}(j)$ is an object in $\text{D}_{\text{pro}}(X_{1,\mathrm{Nis}})$ with the following properties:
\begin{proposition}(\cite[Prop. 7.2]{BEK14})\label{propositionBEK7.2}
\begin{enumerate}
\item[(0)] $\Z_{X_.}(0)=\Z$, the constant sheaf in degree $0$.
\item[(1)] $\Z_{X_.}(j)=\mathbb{G}_{m,X_\cdot}[-1]$.
\item[(2)] $\Z_{X_.}(j)$ is supported in degrees $\leq j$ and in $[1,j]$ if $j\geq 1$ and if the Beilinson-Soul\'e conjecture holds. 
\item[(3)] $\Z_{X_.}(j)\otimes_{\Z}^L\Z/p^.= \mathcal{S}_{X_.}(j)_\mathrm{Nis}$ in $\mathrm{D}_{\mathrm{pro}}(X_{1,\mathrm{Nis}})$.  
\item[(4)] $\mathcal{H}^j(\Z_{X_.}(j))=\mathcal{K}^M_{j,X_.}$ in $\mathrm{Sh}_\mathrm{pro}(X_{1,\mathrm{Nis}})$.
\item[(5)] There is a canonical product structure $\Z_{X_.}(j)\otimes^L_\Z \Z_{X_.}(j')\r \Z_{X_.}(j+j')$.
\end{enumerate}
\end{proposition}

We now start the proof of Proposition \ref{propositionrelationmilnorktosyntomic} proving the following lemmas:
\begin{lemma}\label{lemmamilnorksynnis}
Let $j<p$. Then the map
$$H^{d}_{\mathrm{Nis}}(X_1,\mathcal{K}^M_{j,X_\cdot}\otimes \Z/p^\cdot)\r H^{d+j}_{\mathrm{Nis}}(X_1,\mathcal{S}_{X_\cdot}(j))$$
is an isomorphism in the category of pro-abelian groups.
\end{lemma}
\begin{proof}
By properties (2)-(4) of Proposition \ref{propositionBEK7.2} we have that $\mathcal{K}^M_{j,X_{\cdot}}\otimes \Z/p^\cdot\cong \mathcal{H}^j(\mathcal{S}_\cdot(j)_{\mathrm{Nis}})$. Let us be more precise:
$$\mathcal{K}^M_{j,X_{\cdot}}\otimes \Z/p^\cdot\stackrel{(4)}{\cong} \mathcal{H}^j(\Z_{X_.}(j))\otimes \Z/p^\cdot \stackrel{(2)}{\cong} \mathcal{H}^j(\Z_{X_.}(j)\otimes^\mathbb{L} \Z/p^\cdot) \stackrel{(3)}{\cong} \mathcal{H}^j(\mathcal{S}_\cdot(j)_{\mathrm{Nis}}).$$
For the isomorphism in the middle consider the short exact sequence 
$$0\r \mathcal{H}^j(\Z_{X_.}(j))\otimes \Z/p^\cdot \r \mathcal{H}^j(\Z_{X_.}(j)\otimes^\mathbb{L} \Z/p^\cdot)\r \mathcal{H}^{j+1}(\Z_{X_.}(j))[p^\cdot]\r 0 $$
in which the term on the right vanishes by $(2)$.
This implies that $$H^{d}_{\mathrm{Nis}}(X_1,\mathcal{K}^M_{j,X_\cdot}\otimes \Z/p^\cdot)\cong H^d_{\mathrm{Nis}}(X_1,\mathcal{H}^j(\mathcal{S}_\cdot(j))).$$ The hypercohomology spectral sequence 
$$E_2^{pq}= H^p_{\mathrm{Nis}}(X_1,\mathcal{H}^q(\mathcal{S}_r(j)))\Rightarrow \mathbb{H}^{p+q}_{\mathrm{Nis}}(X_1,\mathcal{S}_r(j))$$
together with the Nisnevich cohomological dimension of $X_1$ and the concentration of $\mathcal{S}_r(j)_{\mathrm{Nis}}$ in degrees $\leq j$ implies that $H^{d}_{\mathrm{Nis}}(X_1,\mathcal{H}^j(\mathcal{S}_r(j)))\cong H^{d+j}_{\mathrm{Nis}}(X_1,\mathcal{S}_r(j))$ and therefore that $H^{d}_{\mathrm{Nis}}(X_1,\mathcal{K}^M_{j,X_\cdot}\otimes \Z/p^\cdot)\r H^{d+j}_{\mathrm{Nis}}(X_1,\mathcal{S}_\cdot(j))$.
\end{proof}

\begin{lemma}\label{derhamwittetalnis}
The natural map
$$\rho_{X_\Nis}^{d,q}: H_{\mathrm{Nis}}^{2d-q}(X_1,W_r\Omega_{X_1,\mathrm{log}}^d[-d])\r H_{\et}^{2d-q}(X_1,W_r\Omega_{X_1,\mathrm{log}}^d[-d])$$
is an isomorphism for $q\in \{0,1\}$. 
\end{lemma}
\begin{proof}
Let $KH^0_a(X_1,\Z/p^r\Z)$ denote the so called Kato homology groups, i.e. the homology in degree $a$ of the complex $C_{p^r}^0$ defined in \cite{Ka86}.
By \cite[Lem. 6.2]{JS} (see also \cite[Sec. 9]{KeS12}) there is a long exact sequence
\begin{equation*}
\begin{split}
...\r KH^{0}_{q+2}(X_1,\Z/p^r\Z)\r \CH^d(X_1,q;\Z/p^r\Z)\xrightarrow{\rho_{X_\Zar}^{d,q}} H_{\text{\'et}}^{2d-q}(X_1,\Z/p^r\Z(d))\r \\
KH^{0}_{q+1}(X_1,\Z/p^r\Z)\r \CH^d(X_1,q-1;\Z/p^r\Z)\r H_{\text{\'et}}^{2d-q+1}(X_1,\Z/p^r\Z(d)) \r ...
\end{split} 
\end{equation*}
where $\Z/p^r\Z(d)=W_r\Omega^d_{X_1,\mathrm{log}}[-d]$. We first identify the group $\CH^d(X_1,q;\Z/p^r\Z)$ with $H_{\text{Nis}}^{2d-q}(X_1,W_r\Omega^d_{X_1,\mathrm{log}}[-d])$ for $q=0,1$.
Consider the spectral sequence 
$$ ^{\CH}E^{p,q}_1(X_1)=\oplus_{x\in X_1^{(p)}}\CH^{d-p}(\text{Spec}k(x),-p-q,\Z/p^r\Z)\Rightarrow \CH^d(X_1,-p-q,\Z/p^r\Z)$$
from \cite[Sec. 10]{Bl86} and note that 
$$\CH^{a}(\text{Spec}k(x),a,\Z/p^r\Z)\cong K^M_a(k(x))/p^r\cong W_r\Omega^a_{k(x),\text{log}}$$
for all $a\geq 0$. The first isomorphism follows from \cite[Thm. 4.9]{NS89} (see also \cite{To92}) and the fact that $\CH^{a}(\text{Spec}k(x),a,\Z/p^r\Z)\cong \CH^{a}(\text{Spec}k(x),a)\otimes \Z/p^r\Z$. The second isomorphism follows from the Bloch-Gabber-Kato theorem (see \cite{BK86}). This implies the identification since $\CH^0(k(x),1)=0$ and since
$$\bigoplus_{x\in X^0} i_{x*}W_r\Omega^d_{k(x),\text{log}}\r \bigoplus_{x\in X^1} i_{x*}W_r\Omega^{d-1}_{k(x),\text{log}}\r ...\r \bigoplus_{x\in X^d} i_{x*}W_r\Omega^{0}_{k(x),\text{log}}$$
is a (Gersten-)resolution for the sheaf $W_r\Omega^d_{X_1,\mathrm{log}}$ considered in the Zariski topology (see \cite{GS88}) and therefore also in the Nisnevich topology. In particular $H_{\mathrm{Zar}}^{i}(X_1,W_r\Omega_{X_1,\mathrm{log}}^d)\cong H_{\mathrm{Nis}}^{i}(X_1,W_r\Omega_{X_1,\mathrm{log}}^d)$ for all $i\geq 0$. Note furthermore that $\rho_{X_\Zar}^{d,q}$ factors through $\rho_{X_\Nis}^{d,q}$ since it comes from the change of sites $\epsilon:X_\et\r X_\Nis\r X_\Zar$. In fact, Nisnevich and Zariski motivic cohomology coincide.

Now the Kato homology groups $KH^{0}_{i}(X_1,\Z/p^r\Z)$ vanishes for $1\leq i \leq 4$ by \cite[Thm. 0.3]{JS} (see also \cite[Thm. 8.1]{KeS12}) which implies the lemma.
\end{proof}

\begin{lemma}\label{lemmasyntomicnisetale} Let $j<p$. Then
$$H^{j+d}_{\mathrm{Nis}}(X_1,\mathcal{S}_{X_\cdot}(j))\r H^{j+d}_{\et}(X_1,\mathcal{S}_{X_\cdot}(j))$$
is an isomorphism for all $j\geq d$.
\end{lemma}
\begin{proof}
By \cite[Thm 5.4]{BEK14} we have an exact triangle
$$p(j)\Omega_{X_\cdot}^{\leq j}[-1]\r S_{X_\cdot}(j)_{\text{Nis}}\r W_{\cdot}\Omega^j_{X_1,\text{log}}[-j]\xrightarrow{[1]}..$$  
in $\text{D}_{\text{pro}}(X_1)_{\text{Nis}}$
which comes from the exact triangle
$$p(j)\Omega_{X_\cdot}^{\leq j}[-1]\r S_{X_\cdot}(j)_{\text{\'et}}\r W_{\cdot}\Omega^j_{X_1,\text{log}}[-j]\xrightarrow{[1]}..$$  
in $\text{D}_{\text{pro}}(X_1)_{\text{\'et}}$ by applying the functor $\tau_{\leq j}\circ R\epsilon_*$. This induces the following commutative diagram with exact rows:
$$\begin{xy} 
  \xymatrix{
   H^{d+j-1}_{\text{Nis}}(X_1,W_{\cdot}\Omega^j_{X_1,\text{log}}[-j]) \ar[r]^{} \ar[d]_{\alpha} &  H^{d+j}_{\text{Nis}}(X_1,p(j)\Omega_{X_\cdot}^{\leq j}[-1]) \ar[d]^{\beta} \ar[r]^{} \ar[d]_{} &  H^{d+j}_{\text{Nis}}(X_1,S_{X_\cdot}(j)_{\text{Nis}}) \ar[d]^{} \\
    H^{d+j-1}_{\text{\'et}}(X_1,W_{\cdot}\Omega^j_{X_1,\text{log}}[-j]) \ar[r]^{}    &  H^{d+j}_{\text{\'et}}(X_1,p(j)\Omega_{X_\cdot}^{\leq j}[-1]) \ar[r]^{}  & H^{d+j}_{\text{\'et}}(X_1,S_{X_\cdot}(j)_{\text{\'et}})
  }
\end{xy} $$

$$\begin{xy} 
  \xymatrix{
  \ar[r]^{} & H^{d+j}_{\text{Nis}}(X_1,W_{\cdot}\Omega^j_{X_1,\text{log}}[-j]) \ar[r]^{} \ar[d]_{\gamma} &  H^{d+j+1}_{\text{Nis}}(X_1,p(j)\Omega_{X_\cdot}^{\leq j}[-1]) \ar[d]^{\delta} \ar[d]_{}  \\
   \ar[r]^{} & H^{d+j}_{\text{\'et}}(X_1,W_{\cdot}\Omega^j_{X_1,\text{log}}[-j]) \ar[r]^{}    &  H^{d+j+1}_{\text{\'et}}(X_1,p(j)\Omega_{X_\cdot}^{\leq j}[-1])  
  }
\end{xy} $$
Now $\alpha$ and $\gamma$ are isomorphisms by Lemma \ref{derhamwittetalnis} and the fact that $W_{\cdot}\Omega^j_{X_1,\text{log}}[-j]=0$ for $j>d$. The maps $\beta$ and $\delta$ are isomorphisms since $p(j)\Omega_{X_\cdot}^{\leq j}[-1]$ is a complex of coherent sheaves. The result follows by the five-lemma.
\end{proof}

\begin{proposition}\label{propositionidentificationmilnorksyntomic} Let $p>j$ and $j\geq d$. Then the map
$$"\mathrm{lim}_n" H^{d}_{\mathrm{Nis}}(X_1,\mathcal{K}^M_{j,X_n}/p^r)\r H^{d+j}_{\et}(X_1,\mathcal{S}_r(j)) $$
is an isomorphism of pro-abelian groups. 
\end{proposition}
\begin{proof}
It follows from Lemma \ref{lemmamilnorksynnis} and Lemma \ref{lemmasyntomicnisetale} that
$$H^{d}_{\mathrm{Nis}}(X_1,\mathcal{K}^M_{j,X_\cdot}\otimes \Z/p^\cdot)\r H^{d+j}_{\et}(X_1,\mathcal{S}_\cdot(j)).$$
Tensoring with $\Z/p^r$ gives the desired result: Since $\otimes\Z/p^r$ is right exact and the cohomology group on the left is taken in the top degree, it follows that $"\mathrm{lim}_n" H^{d}_{\mathrm{Nis}}(X_1,\mathcal{K}^M_{j,X_n}/p^r)\cong H^{d}_{\mathrm{Nis}}(X_1,\mathcal{K}^M_{j,X_\cdot}\otimes \Z/p^\cdot)\otimes\Z/p^r$. For the right side note first that $\mathcal{S}_\cdot(j)_{\et}\otimes \Z/p^r \cong \mathcal{S}_r(j)_{\et}$. This follows for example from \cite[Thm. 1.3]{Ge04}. Now consider the short exact sequence 
$$0\r H^{d+j}_{\et}(X_1,\mathcal{S}_{X_\cdot}(j))/p^r\r H^{d+j}_{\et}(X_1,\mathcal{S}_{X_\cdot}(j)\otimes^\mathbb{L} \Z/p^r)\r H^{d+j+1}_{\et}(X_1,\mathcal{S}_{X_\cdot}(j))[p^r]\r 0$$
and note that $H^{d+j+1}_{\et}(X_1,\mathcal{S}_{X_\cdot}(j))[p^r]=0$. For $j\geq d+1$ this is clear since in that case $H^{d+j+1}_{\et}(X_1,\mathcal{S}_{X_\cdot}(j))=0$. For $j=d$ we have that
$$H^{2d+1}_{\et}(X_1,\mathcal{S}_{X_\cdot}(d))\cong H^{2d+1}_{\et}(X_1,W_{\cdot}\Omega^d_{X_1,\text{log}}[-d])\cong KH^0_0(X_1,\Z/p^\cdot\Z)\cong \Z/p^\cdot\Z.$$
Here the first isomorphism follows from the diagram in the proof of Lemma \ref{lemmasyntomicnisetale}, the second from the long exact sequence in the proof of Lemma \ref{derhamwittetalnis} extended to the right and the third from the Kato conjectures (see \cite[Thm. 8.1]{KeS12}). But $\Z/p^\cdot\Z[p^r]=0$ since for every $\Z/p^m\Z[p^r]=0$ we can find an $m'$ (f.e. $m'=m+r$) such that $\Z/p^{m'}\Z[p^r]\r \Z/p^m\Z[p^r]$ is the zero map.
\end{proof}

\begin{proposition/definition}\label{diagrammintroduction2proof}
Diagram (\ref{diagrammintroduction2}) of the introduction commutes.
\end{proposition/definition}
\begin{proof} Let $X_K$ be the generic fiber of $X$. Let $i:X_1\hookrightarrow X$ and $j:X_K\hookrightarrow X$ be the canonical inclusions.

The exact Kummer sequence 
$$0\r \mu_{p^n}\r \O_{X_K}^\times\xrightarrow{p^n} \O_{X_K}^\times\r 0$$
on $X_{K,\et}$ induces an exact sequence 
$$j_*\O_{X_K}^\times\xrightarrow{p^n}j_*\O_{X_K}^\times\r R^1j_*\mu_{p^n}\r 0$$
on $X_\et$ which induces a Galois symbol map 
$$j_*\mathcal{K}^M_{q,X_K}\r R^qj_*(\Z/p^r\Z(q))$$
(see \cite[(1.2)]{BK86}).
This map induces a map  
$$\mathcal{K}^M_{q,X}\r i_*\text{ker}(\sigma^q_{X,r}:i^*R^qj_*(\Z/p^r\Z(q))\r W_r\Omega^{q-1}_{X_1,\log})\cong \H^q(\mathcal{T}_r(q))$$
in the \'etale topology.
For the definition of $\sigma^q_{X,r}$ see \cite[Sec. 3.2]{Sa07} and for the isomorphism on the right see \cite[Def. 4.2.4]{Sa07}. Furthermore, $\text{ker}(\sigma^q_{X,r}:i^*R^qj_*(\Z/p^r\Z(q))\r W_r\Omega^{q-1}_{X_1,\log})\cong \H^q(\mathcal{S}_r(q))$ by \cite{Ka87}. We have the following commutative diagram in the \'etale topology:
$$\begin{xy}
  \xymatrix{
      \mathcal{K}^M_{q,X}/p^r \ar[d]^-{} \ar[r]  & i_*\mathcal{K}^M_{q,X_{r+1}}/p^r \ar[d]^-{(*)} \\
  \H^q(\mathcal{T}_r(q))   \ar[r]^-{\cong} &i_*\H^q(\mathcal{S}_r(q))
  }
\end{xy} $$
Here $(*)$ is induced by Kato's syntomic regulator map (\cite[Sec. 3]{Ka87}) and the commutativity follows from \cite[Lem. 4.2]{Ka87}.

Taking cohomology groups we get the following commutative diagram:

$$\begin{xy}
  \xymatrix{
      H^d_\Nis(X, \mathcal{K}^M_{q,X}/p^r) \ar[d]^-{} \ar[r]  & H^d_\Nis(X_1, \mathcal{K}^M_{q,X_{r+1}}/p^r) \ar[d]^-{(*)} \\
  H^d_\Nis(X, \H^q(\mathcal{T}_r(q)_\Nis))  \ar[d]^-{\cong}  \ar[r]^-{\cong} & H^d_\Nis(X_1,\H^q(\mathcal{S}_r(q)_\Nis)) \ar[d]^-{\cong}\\
  H^{d+q}_\Nis(X, \mathcal{T}_r(q)_\Nis)   \ar[d] \ar[r]^-{\cong} & H^{d+q}_\Nis(X_1,\mathcal{S}_r(q)_\Nis) \ar[d]\\
  H^{d+q}_\et(X, \mathcal{T}_r(q))    \ar[r]^-{\cong} & H^{d+q}_\et(X_1,\mathcal{S}_r(q)),
  }
\end{xy} $$
where $\mathcal{T}_r(q)_\Nis:= \tau_{\leq q}R\epsilon_*\mathcal{T}_r(q)_\et$. The lower horizontal isomorphism follows from proper base change and the fact that $i^*\mathcal{T}_r(n)\cong \mathcal{S}_r(n)$ if $p>n+1$. For the isomorphism on the right see the proof of Proposition \ref{lemmamilnorksynnis}. The change of sites $\epsilon:X_\Nis\r X_\Zar$ now gives the result.
\end{proof}

As a corollary we get the following result:
\begin{corollary}\label{corollarysurjectivitytate}
Let $j+1<p$. Then the cycle class map 
$$\varrho^{j,j-d}_{p^r}:\CH^{j}(X,j-d,\Z/p^r\Z)\r H^{d+j}_{\et}(X,\mathcal{T}_r(j))$$
is surjective for all $j\geq d$.
\end{corollary}
\begin{proof}
By Corollary \ref{corollaryvanishingidelicH2} and Remark \ref{zarnismotivicthickenings} the map $$res:\CH^{j}(X,j-d,\Z/p^r\Z)\xrightarrow{} "\text{lim}_n"H^{d}_\mathrm{Zar}(X_1,\mathcal{K}^M_{j,X_n}/p^r)\cong H^{d}_\mathrm{Nis}(X_1,\mathcal{K}^M_{j,X_n}/p^r)$$ is surjective for $j-1<p$. By Proposition \ref{propositionidentificationmilnorksyntomic} we have that $"\text{lim}_n" H^{d}_\mathrm{Nis}(X_1,\mathcal{K}^M_{j,X_n}/p^r)\cong H^{d+j}_{\text{\'et}}(X_1,\mathcal{S}_r(j))$ for $j<p$. Furthermore, for $j+1<p$ we have that $H^{d+j}_{\text{\'et}}(X_1,\mathcal{S}_r(j))\cong H^{d+j}_{\text{\'et}}(X,\mathcal{T}_r(j))$ (see \cite[Sec. 1.4]{Sa07}). The result now follows from the commutativity of (\ref{diagrammintroduction2}).
\end{proof}

\begin{remark} As we noted in the introduction, Saito and Sato show in \cite{SS14} that the the cycle class map
$$\varrho_{p^r}^{d,0}:\CH^d(X)/p^r\r H^{2d}_{\mathrm{\et}}(X,\mathcal{T}_r(d))$$
defined in \cite[Cor. 6.1.4]{Sa07} is surjectiv for $X$ a regular scheme which is proper, flat, of finite type and which has semistable reduction over $\O_K$, where $\O_K$ is the ring of integers in a $p$-adic local field $K$. We expect that this map coincides with the map defined in Proposition \ref{diagrammintroduction2proof}. 
\end{remark} 
 
Finally we note the following injectivity result for curves:
\begin{proposition}\label{propositioninjectivityreldim1}
Let $X$ be smooth projective of relative dimension $1$ over a $p$-adic local ring $A$. Then the cycle class map 
$$\varrho^{1,0}_{p^r}:\CH_{1}(X)/p^r\r H^{2}_{\et}(X,\mathcal{T}_r(1))$$
is injective.
\end{proposition}
\begin{proof}
This follows immediately from the spectral sequence 
$$E_1^{u,v}=\oplus_{x\in X^u}H^{v+u}_{\et,x}(X,\mathcal{T}_r(d))\Rightarrow H^{v+u}_{\et}(X,\mathcal{T}_r(d))$$
since by absolute cohomological purity and the purity property of $p$-adic \'etale Tate twists $E_2^{1,1}\cong \CH_{1}(X)/p^r$ (see \cite{Sa05}).
\end{proof}

Under the assumptions of Corollary \ref{conjrel1} we therefore get a sequence of isomorphisms
$$H^{2}_{\et}(X,\mathcal{T}_r(1)) \xleftarrow{\cong} \CH_{1}(X)/p^r\xrightarrow{\cong} "\mathrm{lim}_n"H^1(X_1,\mathcal{K}^M_{X_n,1}/p^r).$$
The isomorphism on the left, induced by $\varrho^{1,0}_{p^r}$, follows from Proposition \ref{propositioninjectivityreldim1} and the theorem of Saito and Sato mentioned in the introduction and the isomorphism on the right follows from Corollary \ref{conjrel1}. It would be interesting to have a similar result for $\CH^2(X,1,\Z/p^r\Z)$.

\bibliographystyle{siam}
\bibliography{Bibliografie} 

\end{document}